\documentclass[12pt]{amsart}

\usepackage[utf8]{inputenc}

\usepackage[margin=1.3in]{geometry}

\usepackage{amsmath}
\usepackage{amsthm}
\usepackage{amssymb}
\usepackage{color}
\input xy
\xyoption{all}
\usepackage[all,cmtip]{xy}
\usepackage{tcolorbox}
\usepackage{enumitem}
\usepackage{graphicx}

\theoremstyle{definition}
\newtheorem{theorem}{Theorem}[section]

\newtheorem{definition}[theorem]{Definition}
\newtheorem{example}[theorem]{Example}
\newtheorem{c-example}[theorem]{Counter-example}
\newtheorem{Lemma}[theorem]{Lemma}
\newtheorem{corollary}[theorem]{Corollary}
\newtheorem{Prop}[theorem]{Proposition}
\newtheorem{remark}[theorem]{Remark}

\numberwithin{equation}{section}

\usepackage{hyperref}
\usepackage[pagewise]{lineno}

\newcommand{\Cal}[1]{{\mathcal #1}}

\newcommand{\paral}[1]{\ar@<0.3ex>[#1] \ar@<-0.3ex>[#1]}

\newcommand{\pushoutcorner}[1][dr]{\save*!/#1+1.2pc/#1:(1,-1)@^{|-}\restore}
\newcommand{\pullbackcorner}[1][dr]{\save*!/#1-1.2pc/#1:(-1,1)@^{|-}\restore}

\newcommand{\Mod}[1]{\operatorname{Mod}#1}

\DeclareMathOperator{\Hom}{Hom}

\DeclareMathOperator{\Triv}{\mathsf{Triv}}

\renewcommand{\mod}{\operatorname{mod}}

\newdir{ >}{{}*!/-7pt/@{>}}

\title{Building pretorsion theories from torsion theories}
\author[F. Campanini]{Federico Campanini}
\address{Universit\'e catholique de Louvain, Institut de Recherche en Math\'ematique et Physique, 1348 Louvain-la-Neuve, Belgium}
\email{federico.campanini@uclouvain.be}

\author[F. Fedele]{Francesca Fedele}
\address{School of Mathematics, University of Leeds, Leeds, LS2 9JT, United Kingdom}
\email{f.fedele@leeds.ac.uk}

 \thanks{ The first named author is a postdoctoral researcher of the Fonds de la Recherche Scientifique - FNRS. The second named author is supported by the EPSRC grant ``Combinatorial Representation Theory: Algebra and its interfaces with geometry, topology and combinatorics". }


\begin{document}

\begin{abstract}
Torsion theories play an important role in abelian categories and they have been widely studied in the last sixty years. In recent years, with the introduction of pretorsion theories, the definition has been extended to general (non-pointed) categories. Many examples have been investigated in several different contexts, such as topological spaces and topological groups, internal preorders, preordered groups, toposes, V-groups, crossed modules, etc.
In this paper, we show that pretorsion theories naturally appear also in the ``classical" framework, namely in abelian categories.
We propose two ways of obtaining pretorsion theories starting from torsion theories. The first one uses ``comparable'' torsion theories, while the second one extends a torsion theory with a Serre subcategory. We also give a universal way of obtaining a torsion theory from a given pretorsion theory in additive categories. We conclude by providing several applications in module categories, internal groupoids, recollements and representation theory.
\end{abstract}
\maketitle

\section*{Introduction}
Pretorsion theories were defined in \cite{FF, FFG} as ``non-pointed torsion theories", where the zero object and the zero morphisms are replaced by a class of “trivial objects” and an ideal of ``trivial morphisms”, respectively. This notion generalises many concepts of torsion theory introduced and investigated by several authors in pointed and multi-pointed categories \cite{D, BG, CDT, GJ, JT}. Pretorsion theories appear in several different contexts, such as topological spaces and topological groups \cite{FFG}, internal preorders \cite{FF, FFG2, BCG, BCG2}, categories \cite{BCG3, BCGT, Xarez}, preordered groups \cite{GM}, V-groups \cite{Michel}, crossed modules, etc.

In this paper, we present two ways of obtaining pretorsion theories starting from torsion theories, so that many new examples of pretorsion theories can be given in pointed categories.
Lattices and chains of torsion theories are widely studied topics and they are the perfect framework for applying the first result we prove, where two ``comparable'' torsion theories are used to build a pretorsion theory, as follows.

\medskip

\noindent\textbf{Theorem~\ref{main}.}
Let $\Cal C$ be a pointed category and consider two torsion theories $(\Cal T_1,\Cal F_1)$ and $(\Cal T_2, \Cal F_2)$ in it. Then, the following conditions are equivalent:
\begin{enumerate}
    \item
    $\Cal T_2 \subseteq \Cal T_1$;
    \item
    $\Cal F_1\subseteq \Cal F_2$;
    \item
    $(\Cal T_1, \Cal F_2)$ is a pretorsion theory.
\end{enumerate}
Moreover, if these conditions hold, then $\Cal T_1= \Cal T_2 \ast \Cal Z$ and $\Cal F_2=\Cal Z \ast \Cal F_1$, where $\Cal Z:= \Cal T_1 \cap \Cal F_2$.
\medskip

The second method we present to build pretorsion theories consists of ``extending'' a torsion theory with a Serre subcategory, that is, a full subcategory closed under subobjects, quotients and extensions.

\medskip
\noindent\textbf{Theorem~\ref{extension_Serre_2}.}
Let $\Cal C$ be a pointed category in which every morphism admits an (epi, mono)-factorisation. Assume that $\Cal C$ has pullbacks and pushouts which preserve normal epimorphisms and normal monomorphisms respectively. Let $(\Cal U, \Cal V)$ be a torsion theory in $\Cal C$ and let $\Cal S$ be a monocoreflective and epireflective Serre subcategory of $\Cal C$. Then, the pair $(\Cal T, \Cal F)=(\Cal U \ast \Cal S, \Cal S \ast \Cal V)$ is a pretorsion theory with class of trivial objects $\Cal S$.
\medskip

All the assumptions on $\Cal C$ in the above statement hold for abelian categories and more generally for almost abelian categories in the sense of Rump \cite{R} (e.g. the category of topological (Hausdorff) abelian groups). In particular, any torsion theory of an abelian category can be extended by any bilocalising Serre subcategory, as shown in Example~\ref{recollements}.

\medskip
After these results, we show how to obtain a torsion theory from a given pretorsion theory in an additive category. This construction is universal, and it is the analogue of the universal stable category provided in \cite{BCG3} for lextensive categories.

\medskip
\noindent\textbf{Theorem \ref{stable_theorem}.}
    Let $(\Cal T, \Cal F)$ be a pretorsion theory in an additive category $\Cal C$ with class of trivial objects $\Cal Z$. The quotient functor $\Sigma\colon \Cal C \to \Cal C / \Cal Z$ satisfies the following properties:
    \begin{enumerate}
        \item
        $\Sigma$ sends trivial objects and trivial morphisms into the zero object and zero morphisms respectively;
        \item 
        $\Sigma$ is an additive torsion theory functor;
        \item 
        $(\Sigma(\Cal T), \Sigma(\Cal F))$ is a torsion theory in $\Cal C/ \Cal Z$;
    \end{enumerate}
    Moreover, if $G\colon \Cal C \to \Cal D$ is a functor into a pointed category $\Cal D$ satisfying conditions (1), (2) and (3), then there is a unique functor $H \colon \Cal C / \Cal Z \to \Cal D$ making the following diagram commute
    $$
    \xymatrix{
    \Cal C \ar[rr]^\Sigma \ar[dr]_G & & \Cal C / \Cal Z \ar@{..>}[ld]^H\\
    & \Cal D &
    }.
    $$
\medskip

The paper is organised as follows. In Section~\ref{section_preliminaries}, we recall some key background on torsion and pretorsion theories. Sections ~\ref{Pretorsion theories from pairs of torsion theories}, \ref{Pretorsion theories as extensions of a torsion theory with a Serre subcategory} and \ref{stable_category} respectively present and prove the above three theorems. Finally, Section~\ref{Section_examples} presents applications of the results in various examples, using lattices and chains of torsion theories, and recollements of abelian categories.

\section{Torsion theories and pretorsion theories}\label{section_preliminaries}

Throughout the paper, we will widely use the following notation and terminology.
\begin{itemize} 
\item A subcategory $\Cal B$ of a given category $\Cal C$ is \textit{closed under subobjects [resp. quotients]} if for every monomorphism [resp. epimorphism] with codomain [resp. domain] in $\Cal B$, then also the domain [resp. codomain] is in $\Cal B$.
\item We say that a morphism $f$ \textit{admits an (epi,mono)-factorisation} if it can be written as $f=me$ with $m$ monomorphism and $e$ epimorphism. The codomain of $e$ will be called an \textit{image} of $f$.
\item An epimorphism $f$ is \textit{extremal} if whenever $f=mg$ with $m$ a monomorphism, then $m$ is an isomorphism. An epimorphism is said to be \textit{normal} if it is the cokernel of some morphism. Any normal epimorphism is an extremal epimorphism. \textit{Extremal} and \textit{normal monomorphisms} are defined dually.
\end{itemize}

\medskip

The notion of torsion theory for abelian categories was introduced in \cite{D} by Dickson in 1966 and serves as a standard tool in module
theory and in abelian category theory (see for instance \cite{S}). Here we recall the definition.

\begin{definition}
A \textit{torsion theory} in an abelian category $\Cal C$ is a pair $(\Cal T, \Cal F)$ of full subcategories of $\Cal C$ closed under isomorphisms, such that:
\begin{itemize}
\item
$\Hom(T, F)=0$, for every $T\in\Cal T$ and $F\in\Cal F$;
\item
for every object $X$ in $\Cal C$, there is a short exact sequence
$$\xymatrix{0 \ar[r] & T_X \ar[r]^f &  X \ar[r]^g &  F_X \ar[r] & 0}$$ with $T_X\in\Cal T$ and $F_X\in\Cal F$.
\end{itemize}
\end{definition}
It is worth noting that the notion of torsion theory makes sense in any pointed category, and in fact several authors studied torsion theories out of the abelian case (see for example \cite{D, BG, CDT, GJ, JT}).
More recently, in \cite{FF, FFG} pretorsion theories were defined  as ``non-pointed torsion theories", where the zero object and the zero morphisms are replaced by a class of “trivial objects” and an ideal of ``trivial morphisms”, respectively, as follows.

Let $\Cal C$ be an arbitrary category and fix a class $\Cal Z$ of objects of $\Cal C$, that we shall call \emph{the class of trivial objects}. A morphism $f\colon A\to A'$ in $\Cal C$ is \textit{$\Cal Z$-trivial} if it factors through an object of $\Cal Z$. Given any two objects $X$ and $Y$, we denote by $\Triv(X,Y)$ the class of $\Cal Z$-trivial morphisms from $X$ to $Y$ and by $\Triv$ the class of all $\Cal Z$-trivial morphisms in $\Cal C$. Notice that $\Triv$ is an ideal of morphisms in the sense of Ehresmann \cite{Ehr}, that is, for every pair of composable morphisms $f$ and $g$ in $\Cal C$, $fg\in \Triv$ whenever $f$ or $g$ is in $\Triv$. Hence, it is possible to consider the notions of $\mathcal Z$-kernel and $\mathcal Z$-cokernel, defined by replacing, in the definition of kernel and cokernel, the ideal of zero morphisms with the ideal of trivial morphisms induced by the class $\mathcal Z$ as follows.

\begin{definition}
A morphism $\varepsilon\colon X\to A$ in $\Cal C $ is a \emph{$\Cal Z$-kernel} of $f\colon A \to A'$ if $f\varepsilon$ is a $\Cal Z$-trivial morphism and, whenever $\lambda \colon Y\to A$ is a morphism in $\Cal C$ and $f\lambda$ is $\Cal Z$-trivial, there exists a unique morphism $\lambda'\colon Y\to X$ in $\Cal C$ such that $\lambda=\varepsilon\lambda'$.
The notion of \emph{$\Cal Z$-cokernel} is defined dually. A sequence $A\overset{f}{\to}B\overset{g}{\to}C$ is called a \emph{short $\Cal Z$-exact sequence} if $f$ is a $\Cal Z$-kernel of $g$ and $g$ is a $\Cal Z$-cokernel of $f$.
\end{definition}

It can be easily seen that $\Cal Z$-kernels and $\Cal Z$-cokernels, whenever they exist, are unique up to isomorphism and they are monomorphisms and epimorphisms respectively \cite{FFG}.

It is worth mentioning that the notions of kernels, cokernels and short exact sequences with respect to an ideal of morphisms played an important role in the works of Lavendhomme \cite{Lav} and Grandis \cite{Gra1, Gra2}. More recently, this approach has also led to a unification of some results in pointed and non-pointed categorical algebra \cite{GJU}.

\begin{definition}
Let $\Cal T$ and $\Cal F$ be full subcategories of $\Cal C$ closed under isomorphisms.
We say that the pair $(\Cal T,\Cal F)$ is a \emph{pretorsion theory} in $\Cal C$ with class of trivial objects $\Cal Z:=\Cal T\cap \Cal F$, if the following two properties are satisfied:
\begin{itemize}
\item
$\Hom(T, F)=\Triv(T,F)$, for every $T \in \Cal T$ and $F \in \Cal F$;
\item
for every object $X$ of $\Cal C$ there is a short $\Cal Z$-exact sequence
$$\xymatrix{ T_X \ar[r]^f &  X \ar[r]^g &  F_X}$$ with $T_X\in\Cal T$ and $F_X\in\Cal F$.
\end{itemize}
\end{definition}

\begin{remark}
    When $\Cal C$ is pointed and $\Cal T\cap \Cal F=0$, we recover the usual notion of torsion theory.
In particular, the following properties are true also for any ``classical" torsion theory.
\end{remark}

Recall from \cite{FFG} that given a pretorsion theory $(\Cal T, \Cal F)$ in a category $\Cal C$, there are two functors:
\begin{itemize}
\item
a ``torsion functor" $T\colon \Cal C\to \Cal T$ which is the right adjoint of the full embedding $E_\Cal T\colon \Cal T \to \Cal C$ of the torsion subcategory $\Cal T$;
\item
a ``torsion-free functor" $F\colon \Cal C\to \Cal F$ which is the left adjoint of the full embedding $E_\Cal F\colon \Cal F \to \Cal C$ of the torsion-free subcategory $\Cal F$.
\end{itemize}
For every object $X \in \Cal C$ there is a short $\Cal Z$-exact sequence
$$
\xymatrix{TX\ar[r]^-{\varepsilon_X} & X \ar[r]^-{\eta_X} & FX}
$$
where the monomorphism $\varepsilon_X$ is the $X$-component of the counit $\varepsilon$ of the adjunction
$$
\xymatrix{\Cal T \ar@/^5pt/[r]^{E_\Cal T} & \Cal C,\ar@/^5pt/[l]^{T}_\perp}
$$
while the epimorphism $\eta_X$ is the $X$-component of the unit $\eta$ of the adjunction
$$
\xymatrix{\Cal C \ar@/^5pt/[r]^F & \Cal F \ar@/^5pt/[l]^{E_\Cal F}_\perp}
$$

\medskip\

Given a pretorsion theory $(\Cal T, \Cal F)$, the torsion [resp. torsion-free] subcategory is closed under all colimits [resp. limits] existing in $\Cal C$, extremal quotiens [resp. extremal subobjects] and $\Cal Z$-extensions (see \cite[Proposition~4.2]{FFG} and \cite[Lemma~2.1]{BCG3}). Moreover, the following properties hold \cite[Proposition~2.7]{FFG}:
    \begin{itemize} 
    \item
    for any $X \in \Cal C$, if $\Hom(X, \Cal F)=\Triv(X, \Cal F)$, then $X \in \Cal T$;
    \item
    for any $Y \in \Cal C$, if $\Hom(\Cal T, Y)=\Triv(\Cal T, Y)$, then $Y \in \Cal F$.
    \end{itemize}
    
\begin{remark} 
In complete, cocomplete and locally small abelian categories, torsion theories can be equivalently defined by using the previous properties, in the following sense. A pair of subcategories $(\Cal T, \Cal F)$ is a torsion theory if and only if
    \begin{itemize}
        \item
        $\Hom(T,F)=0$ for all $T \in \Cal T, F \in \Cal F$;
        \item
        for any $X \in \Cal C$, if $\Hom(X, F)=0$ for all $F \in \Cal F$, then $X \in \Cal T$;
        \item
        for any $Y \in \Cal C$, if $\Hom(T, Y)=0$ for all $T \in \Cal T$, then $Y \in \Cal F$.
    \end{itemize}
    Moreover, a subcategory $\Cal T$ of $\Cal C$ is a torsion class if and only if $\Cal T$ is closed under quotients, coproducts and extensions \cite[Chapter~VI]{S}. These characterizations fail to be true out of the abelian case for torsion and pretorsion theories \cite{CDT, JT, FFG}.
\end{remark}

\section{Pretorsion theories from pairs of torsion theories}\label{Pretorsion theories from pairs of torsion theories}

Let $\Cal C$ be a pointed category and consider two torsion theories $(\Cal T_1,\Cal F_1)$ and $(\Cal T_2, \Cal F_2)$ in it. For $i=1,2$, let $T_i\colon \Cal C \to \Cal T_i$ and $F_i\colon \Cal C \to \Cal F_i$ denote respectively the torsion and torsion-free functors induced by the torsion theory $(\Cal T_i, \Cal F_i)$. Thus, for every object $X \in \Cal C$ there is a (canonical) short exact sequence
$$
\xymatrix{0\ar[r] & T_i X \ar[r]^-{\varepsilon_{iY}} & X  \ar[r]^{\eta_{iY}} & F_iX \ar[r] & 0}
$$
with $T_i X \in \Cal T_i$ and $F_i X \in \Cal F_i$.

\medskip

Given any two subcategories $\Cal A$ and $\Cal B$ of $\Cal C$, we denote by $\Cal A \ast\Cal B$ the full subcategory of $\Cal C$ whose objects are extensions of an object in $\Cal A$ and an object in $\Cal B$, that is, $X \in \Cal A \ast \Cal B$ if and only if there exists a short exact sequence $0 \to A \to X \to B \to 0$ with $A \in \Cal A$ and $B \in \Cal B$.

\begin{theorem}\label{main}
Let $\Cal C$ be a pointed category and consider two torsion theories $(\Cal T_1,\Cal F_1)$ and $(\Cal T_2, \Cal F_2)$ in it. Then, the following conditions are equivalent:
\begin{enumerate}
    \item
    $\Cal T_2 \subseteq \Cal T_1$;
    \item
    $\Cal F_1\subseteq \Cal F_2$;
    \item
    $(\Cal T_1, \Cal F_2)$ is a pretorsion theory.
\end{enumerate}
Moreover, if these conditions hold, then $\Cal T_1= \Cal T_2 \ast \Cal Z$ and $\Cal F_2=\Cal Z \ast \Cal F_1$, where $\Cal Z:= \Cal T_1 \cap \Cal F_2$.
\end{theorem}

\begin{proof} 
The equivalence of the first two conditions is clear (and well known). Indeed, if $\Cal T_2 \subseteq \Cal T_1$ and $Y \in \Cal F_1$, then $\Hom(X,Y)=0$ for all $X \in \Cal T_1$, hence $\Hom(X,Y)=0$ for all $X \in \Cal T_2$ and so $Y \in \Cal F_2$. The other implication follows by a dual argument.

Assume now that $(\Cal T_1, \Cal F_2)$ is a pretorsion theory and let $X \in \Cal T_2$. Then there is a short $\Cal Z$-exact sequence
$$
\xymatrix{T_X \ar@{>->}[r]^\varepsilon & X \ar@{->>}[r]^\eta & F_X}
$$
where $T_X \in \Cal T_1$ and $F_X\in \Cal F_2$. Since $X \in \Cal T_2$, then $\eta$ is zero and so a trivial morphism. Hence $\varepsilon$ is an isomorphism \cite[Lemma 2.4]{FFG}. It follows that $T_X\cong X \in \Cal T_1$.

Conversely, assume that Condition $(1)$ holds. Let $f\colon X \to Y$ be a morphism with $X \in \Cal T_1$ and $Y \in \Cal F_2$.
Consider the diagram
$$
\xymatrix{
0\ar[r] & T_2X \ar[r]^-{\varepsilon_{2X}} & X \ar[r]^{\eta_{2X}} \ar[d]^f \ar@{.>}[dl]& F_2 X \ar[r] & 0\\
0\ar[r] & T_1 Y \ar[r]^-{\varepsilon_{1Y}} & Y  \ar[r]^{\eta_{1Y}} & F_1Y \ar[r] & 0
}
$$
where the first (resp. second) row is the canonical short exact sequence of $X$ (resp. $Y$) with respect to the torsion theory $(\Cal T_2, \Cal F_2)$ (resp. ($\Cal T_1, \Cal F_1$)). The dotted arrow is induced by the fact that $\eta_{1X}\cdot f=0$. Moreover, $T_1Y$ is a normal subobject of $Y \in \Cal F_2$, hence $T_1Y \in \Cal T_1 \cap \Cal F_2$. Thus $f$ is a trivial morphism.

To conclude, it suffices to show that for every $X \in \Cal C$, the sequence
$$
\xymatrix{
T_1X \ar[r]^-{\varepsilon_{1X}} & X \ar[r]^{\eta_{2X}} & F_2 X
}
$$
is a short $\Cal Z$-exact sequence. Let us prove that $\varepsilon_{1X}$ is the $\Cal Z$-kernel of $\eta_{2X}$ (the ``$\Cal Z$-cokernel part" can be proved dually). From what we have seen above, the composite morphism $\eta_{2X}\cdot\varepsilon_{1X}$ is trivial. Let $g\colon W\to X$ be a morphism such that $\eta_{2X}\cdot g$ is trivial. Applying the functor $F_2$ to the morphism $\eta_{1X}\colon X \to F_1 X$ and using the assumption $\Cal F_1\subseteq \Cal F_2$, we get a commutative diagram
$$
\xymatrix{
 & F_1X \ar@{=}[r] & F_1X \\
T_1X \ar[r]^-{\varepsilon_{1X}} & X \ar[r]^{\eta_{2X}} \ar[u]^{\eta_{1X}} & F_2 X \ar[u]_{F_2\eta_{1X}}\\
  & W \ar[u]^g \ar@{.>}[ul] & 
}
$$
The morphism $\eta_{2X}\cdot g$ is trivial, so in particular it factors through an object in $\Cal T_1$. Thus we have $\eta_{1X}\cdot g=F_2\eta_{1X}\cdot \eta_{2X}\cdot g=0$ and therefore $g$ factors uniquely through $\varepsilon_{1X}$.

For the last assertion, since both  $\Cal Z$ and $\Cal T_2$ are contained in $\Cal T_1$ and $\Cal T_1$ is closed under extensions, we have $\Cal T_1 \supseteq \Cal T_2 \ast \Cal Z$. For the other inclusion, let $X \in \Cal T_1$ and consider its canonical short exact sequence with respect to $(\Cal T_2, \Cal F_2)$
$$
\xymatrix{0\ar[r] & T_2 X \ar[r] & X  \ar[r] & F_2 X \ar[r] & 0}.
$$
Since $\Cal T_1$ is closed under extremal quotients, $F_2X \in \Cal T_1 \cap \Cal F_2=\Cal Z$, hence $\Cal T_1= \Cal T_2 \ast \Cal Z$. The dual argument proves the other equality.
\end{proof}

\begin{remark}
    The class of trivial objects $\Cal T_1 \cap \Cal F_2$ is closed under extensions. It is also closed under subobjects [resp. quotients] if the torsion theory $(\Cal T_1, \Cal F_1)$ is hereditary [resp. $(\Cal T_2, \Cal F_2)$ is cohereditary].
\end{remark}

Lattices and chains of torsion theories are widely studied topics and they are the perfect framework for applying Theorem~\ref{main} in order to get pretorsion theories. We present some applications of our result in Section~\ref{Section_examples}.

\section{Pretorsion theories as extensions of a torsion theory with a Serre subcategory}\label{Pretorsion theories as extensions of a torsion theory with a Serre subcategory}

Let $\Cal C$ be a pointed category. By a Serre subcategory $\Cal S$ of $\Cal C$ we mean a full subcategory of $\Cal C$ closed under subobjects, quotients and extensions.

As in  the previous section, given any two subcategories $\Cal A$ and $\Cal B$ of $\Cal C$, we denote by $\Cal A \ast\Cal B$ the full subcategory of $\Cal C$ whose objects are extensions of an object in $\Cal A$ and an object in $\Cal B$.

\begin{Prop}\label{extension_Serre_1}
Let $(\Cal U, \Cal V)$ be a torsion theory in a pointed category $\Cal C$ in which every morphism admits an (epi, mono)-factorisation, and let $\Cal S$ be a Serre subcategory. Consider the pair $(\Cal T, \Cal F)=(\Cal U \ast \Cal S, \Cal S \ast \Cal V)$. Then:
\begin{enumerate}
\item
$\Cal S = \Cal T \cap \Cal F$;
\item
$\Cal F$ is closed under subobjects (and dually $\Cal T$ is closed under quotients);
\item
$\Hom(\Cal T, \Cal F)=\Triv(\Cal T, \Cal F)$;
\item
For any $X \in \Cal C$, if $\Hom(X, \Cal F)=\Triv(X, \Cal F)$, then $X \in \Cal T$;
\item
For any $Y \in \Cal C$, if $\Hom(\Cal T, Y)=\Triv(\Cal T, Y)$, then $Y \in \Cal F$.
\end{enumerate}
\end{Prop}

\begin{proof}
$(1)$ The inclusion $\Cal S \subseteq \Cal T \cap \Cal F$ is clear. If $X \in \Cal T \cap \Cal F$, then we can consider a commutative diagram with exact rows
$$
\xymatrix{
0 \ar[r] & U'  \ar[r] \ar@{..>}[d] & X \ar[r] \ar@{=}[d] & S' \ar[r] \ar@{..>>}[d]^p & 0\\
0 \ar[r] & U_X \ar[r] \ar@{>..>}[d]^q & X \ar[r] \ar@{=}[d] & V_X  \ar[r] \ar@{..>}[d] & 0\\
0 \ar[r] & S'' \ar[r]        & X \ar[r]            & V'' \ar[r]        & 0
}
$$
where $S', S'' \in \Cal S$, $U', U_X \in \Cal U$ and $V'', V_X \in \Cal V$. The induced dotted arrows $p$  and $q$ are an epimorphism and a monomorphism respectively. Since $\Cal S$ is a Serre subcategory, we can conclude that $X \in \Cal S$.

$(2)$
Let $N \rightarrowtail X$ be a monomorphism with $X \in \Cal F$. Then, we have a commutative diagram with exact rows
$$
\xymatrix{
0 \ar[r] & U_N  \ar[r] \ar@{>..>}[d] & N \ar[r] \ar@{>->}[d] & V_N \ar[r] \ar@{..>}[d] & 0\\
0 \ar[r] & S   \ar[r] & X \ar[r] & V  \ar[r] & 0
}
$$
with $S \in \Cal S$, $U_N \in \Cal U$ and $V, V_N \in \Cal V$. The left dotted arrow is a monomorphism, hence $U_N \in \Cal S$ and $N \in \Cal S \ast \Cal V= \Cal F$. Dually, one can prove that $\mathcal{T}$ is closed under quotients.

$(3)$ Let $f\colon T \to F$ be a morphism from an object in $\Cal T$ to an object in $\Cal F$. By the previous points, the image of $f$ is in $\Cal T \cap \Cal F= \Cal S$.

$(4)$ Let $X \in \Cal C$ be such that $\Hom(X, \Cal F)=\Triv(X, \Cal F)$ and consider the short exact sequence of $X$ associated with the torsion theory $(\Cal U, \Cal V)$:
$$
\xymatrix{0 \ar[r] & U_X \ar[r] & X  \ar[r]^\pi & V_X \ar[r] & 0}
$$
Since $\Cal V \subseteq \Cal F$, $\pi$ is a trivial morphism, hence it can be written as $\pi= \beta \alpha$ where the domain of $\beta$ is in $\Cal S$ and $\beta$ is an epimorphism (because so is $\pi$). Then, $V_X \in \Cal S$ and $X \in \Cal U \ast \Cal S= \Cal T$.

$(5)$ Dual of $(4)$.
\end{proof}

The argument used to prove Proposition~\ref{extension_Serre_1}~(1) shows also that if $X \in \Cal F=\Cal S\ast \Cal V$, then the torsion part $U_X$ of $X$ (w.r.t the torsion theory $(\Cal U, \Cal V)$) is in $\Cal S$. The dual statement holds as well.

\medskip

Notice that the pair $(\Cal U \ast \Cal S, \Cal S \ast \Cal V)$ is not a pretorsion theory in general, as the existence of $\Cal S$-short exact sequences for every object of $\Cal C$ is not guaranteed by the hypothesis (see Example \ref{injective-reduced}).

\begin{theorem}\label{extension_Serre_2}
Let $\Cal C$ be a pointed category in which every morphism admits an (epi, mono)-factorisation. Assume that $\Cal C$ has pullbacks and pushouts which preserve normal epimorphisms and normal monomorphisms respectively. Let $(\Cal U, \Cal V)$ be a torsion theory in $\Cal C$ and let $\Cal S$ be a monocoreflective and epireflective Serre subcategory of $\Cal C$. Then, the pair $(\Cal T, \Cal F)=(\Cal U \ast \Cal S, \Cal S \ast \Cal V)$ is a pretorsion theory with class of trivial objects $\Cal S$.
\end{theorem}

\begin{proof}
By Proposition~\ref{extension_Serre_1}, we only need to show that for every $X \in \Cal C$ there exists an $\Cal S$-short exact sequence $T_X \to X \to F_X$ with $T_X \in \Cal T$ and $F_X \in \Cal F$.

Notice that $\Cal S$ is a monocoreflective and epireflective subcategory of $\Cal C$ if and only if all the $\Cal S$-kernels and $\Cal S$-cokernels of the identity morphisms exist in $\Cal C$ \cite[Section~1.5]{JT}.

Let $X \in \Cal C$ and consider the short exact sequence of $X$ associated with the torsion theory $(\Cal U, \Cal V)$:
$$
\xymatrix{0 \ar[r] & U_X \ar[r] & X  \ar[r] & V_X \ar[r] & 0}
$$
Consider then the commutative diagram
$$
\xymatrix{
0 \ar[r] & U_X \pullbackcorner \ar[r] \ar@{=}[d] & T_X \pullbackcorner \ar[r] \ar[d] & S_X \ar[r] \ar[d] & 0\\
0 \ar[r] & U_X \ar[r] & X  \ar[r] & V_X \ar[r] & 0\\
 & & & V_X \ar@{=}[u] & 
}
$$
obtained as follows:
\begin{itemize}
\item
$S_X\to V_X$ is the $\Cal S$-kernel of the identity of $V_X$;
\item
the right-hand square is a pullback and, in particular, $T_X\to X$ is the $\Cal S$-kernel of $X \to V_X$;
\item
the pullback of the arrows $T_X\rightarrow X \leftarrow U_X$ gives the left-hand square and $U_X \to T_X$ is the kernel of $T_X\to S_X$;
\item
since pullbacks preserve normal epimorphisms, the top row is a short exact sequence and hence $T_X \in \Cal T$.
\end{itemize}
Now, consider the $\Cal S$-cokernel of the identity of $U_X$ and complete the diagram with two pushout squares:
\begin{equation}\label{diagram}
\xymatrix{
0 \ar[r] & U_X \ar[r] \ar@{=}[d] & T_X \ar[r] \ar[d]^{\varepsilon_X}  & S_X \ar[r] \ar[d] & 0\\
0 \ar[r] & U_X \ar[d]\ar[r] & X  \ar[d]^{\eta_X} \ar[r] & V_X \ar[r] & 0\\
0 \ar[r]& S'_X \ar[r] & F_X \pushoutcorner \ar[r] & V_X \pushoutcorner \ar@{=}[u]\ar[r] & 0
}
\end{equation}
We want to show that the middle column is an $\Cal S$-short exact sequence. Since $T_X \in\Cal T$ and $F_X \in \Cal F$, the composite morphism $\eta_X \cdot \varepsilon_X$ is $\Cal S$-trivial by Proposition~\ref{extension_Serre_1}~(3). From the fact that $\varepsilon_X$ is the $\Cal S$-kernel of $X\to V_X$ [resp. $\eta_X$ is the $\Cal S$-cokernel of $U_X \to X$] we get that $\varepsilon_X$ is the $\Cal S$-kernel of $\eta_X$ [resp. $\eta_X$ is the $\Cal S$-cokernel of $\varepsilon_X$].
\end{proof}

\begin{remark}
Examples of categories satisfying the hypothesis of Theorem~\ref{extension_Serre_2} include abelian categories and more generally almost abelian categories in the sense of Rump \cite{R} (e.g. the category of topological (Hausdorff) abelian groups). Notice that the hypothesis on $\Cal C$ about the behaviour of pullbacks can be relaxed: it can be assumed that only pullbacks of a cokernel and an $\Cal S$-kernel exist and preserve cokernels. Dually for pushouts. In particular, we have the following result, as a specialisation of Theorem~\ref{extension_Serre_2}.
\end{remark}

\begin{corollary}\label{extension_Serre_2_corollary}
    Let $\Cal C$ be an abelian category and let $\Cal S$ be a monocoreflective and epireflective subcategory of $\Cal C$. If $(\Cal U, \Cal V)$ is a torsion theory in $\Cal C$, then $(\Cal U \ast \Cal S, \Cal S \ast \Cal V)$ is a pretorsion theory with class of trivial objects $\Cal S$.
\end{corollary}

\section{The stable category of pretorsion theories in additive categories}\label{stable_category}

In this section, we show that for a given pretorsion theory in an additive category $\Cal C$ with class of trivial objects $\Cal Z$, one can construct a quotient category $\Cal C/\Cal Z$ and a quotient functor $\Cal C\to \Cal C/\Cal Z$ that, roughly speaking, sends the pretorsion theory in $\mathcal{C}$ to a classical torsion theory in $\Cal C/\Cal Z$ in a universal way.
This is a result analogous to \cite[Theorem~5.2]{BCG3}, where the construction is provided in the context of lextensive categories (see \cite{BCG3} for the definition of lextensive category and details on the construction).

We start by recalling the definition of torsion theory functor introduced in \cite{BCG2}.

\begin{definition}\label{tt-functor}
 Let $(\mathcal T,  \mathcal F)$ be a pretorsion theory in a category $\Cal A$. If $\Cal B$ is a pointed category with a given torsion theory $({\mathcal T'},  {\mathcal F'})$ in it, we say that a functor $G \colon \Cal A \rightarrow \Cal B$ is a \emph{torsion theory functor} if the following two properties are satisfied:
 \begin{enumerate}
 \item $G (\mathcal T) \subseteq \mathcal T'$ and  $G( \mathcal F) \subseteq \mathcal F'$;
 \item if $TA \rightarrow A \rightarrow FA$ is the canonical short $\mathcal Z$-exact sequence associated with $A$ in the pretorsion theory $({\mathcal T},  {\mathcal F})$, then $$0 \rightarrow GTA \rightarrow GA \rightarrow GFA \rightarrow 0$$ is a short exact sequence in $\Cal B$.
 \end{enumerate}
 \end{definition}

Let $\Cal C$ be an additive category and $\Cal I$ an ideal of morphisms (that is, a class of morphisms such that $fg \in \Cal I$ whenever $f$ or $g$ is in $\Cal I)$. We say that $\Cal I$ is an {\em additive ideal} if $\Cal I(X,Y)$ is a subgroup of $\hom_{\Cal C}(X,Y)$ for every pair of objects $X,Y \in \Cal C$. For an additive ideal $\Cal I$ it is possible to construct a quotient additive category $\Cal C / \Cal I$, whose objects are the same as those of $\Cal C$ and $\hom_{\Cal C / \Cal I}(X,Y):= \hom_{\Cal C}(X,Y)/ \Cal I(X,Y)$. The canonical quotient functor $\Cal C \to \Cal C / \Cal I$ is an additive functor sending all the morphisms in $\Cal I$ into zero morphisms \cite[Section~4.9]{Facchini} and it is universal among all functors with this property.

\begin{Lemma}
    Let $\Cal C$ be an additive category and let $(\Cal T, \Cal F)$ be a pretorsion theory with class of trivial objects $\Cal Z$. The ideal $\Triv$ of $\Cal Z$-trivial morphisms is an additive ideal.
\end{Lemma}

\begin{proof}
    Since $\Cal T$ is closed under products and $\Cal F$ is closed under coproducts, it follows that $\Cal Z=\Cal T \cap \Cal F$ is closed under biproducts (direct-sums). If $f_i\colon X \to Y$ is a morphism factoring through $Z_i \in \Cal Z$ for $i=1,2$, then $f_1+f_2$ factors through $Z_1 \oplus Z_2 \in \Cal Z$.
\end{proof}

When $\Cal I=\Triv$ is the ideal of $\Cal Z$-trivial morphisms of a pretorsion theory, we write $\Cal C/\Cal Z$ in place of $\Cal C/\Triv$ for the quotient category. 

\begin{theorem}\label{stable_theorem}
    Let $(\Cal T, \Cal F)$ be a pretorsion theory in an additive category $\Cal C$ with class of trivial objects $\Cal Z$. The quotient functor $\Sigma\colon \Cal C \to \Cal C / \Cal Z$ satisfies the following properties:
    \begin{enumerate}
        \item
        $\Sigma$ sends trivial objects and trivial morphisms into the zero object and zero morphisms respectively;
        \item 
        $\Sigma$ is an additive torsion theory functor;
        \item 
        $(\Sigma(\Cal T), \Sigma(\Cal F))$ is a torsion theory in $\Cal C/ \Cal Z$;
    \end{enumerate}
    Moreover, if $G\colon \Cal C \to \Cal D$ is a functor into a pointed category $\Cal D$ satisfying conditions (1), (2) and (3), then there is a unique functor $H \colon \Cal C / \Cal Z \to \Cal D$ making the following diagram commute
    $$
    \xymatrix{
    \Cal C \ar[rr]^\Sigma \ar[dr]_G & & \Cal C / \Cal Z \ar@{..>}[ld]^H\\
    & \Cal D &
    }.
    $$
\end{theorem}
    \begin{proof}
        Point $(1)$ is clear from the definition of the quotient category. We also already observed that the quotient functor is additive. Let
        $$
        \xymatrix{TA \ar[r]^\varepsilon & A \ar[r]^\eta & FA}
        $$
        be the short $\Cal Z$-exact sequence of an object $A \in \Cal C$. If we prove that
        $$
        \xymatrix{\Sigma TA \ar[r]^-{\Sigma \varepsilon} & \Sigma A \ar[r]^-{\Sigma \eta} & \Sigma FA}
        $$
        is a short exact sequence in $\Cal C / \Cal Z$, we get at once that $(\Sigma(\Cal T), \Sigma(\Cal F))$ is a torsion theory in $\Cal C/ \Cal Z$ and that $\Sigma$ is a torsion theory functor. Let us show that $\Sigma \varepsilon$ is the kernel of $\Sigma \eta$.

        Let $h \colon X \to A$ be a morphism in $\Cal C$ such that $\Sigma \eta \Sigma h =0$, that is, $\eta h$ is $\Cal Z$-trivial. Thus, there is a unique $f \colon X \to TA$ in $\Cal C$ such that $h=\varepsilon f$ and therefore $\Sigma h= \Sigma \varepsilon \Sigma f$. Let $g\colon X \to TA$ be another morphism such that $\Sigma h=\Sigma \varepsilon \Sigma g$. This means that $\varepsilon(f-g)$ is a trivial morphism, hence it factors through a trivial object $Z \in \Cal Z$. Moreover, since $\varepsilon$ is the $\Cal Z$-kernel of $\eta$, we also have a (unique) morphism $u \colon Z \to A$ making the right-hand triangle in the diagram
        $$
        \xymatrix{
        X \ar[r]^-{f-g} \ar[rd] & TA \ar[r]^\varepsilon & A \\
        & Z \ar[ru] \ar[u]^u&
        }
        $$
        commute. Since $\varepsilon$ is a monomorphism, also the left-hand triangle commutes. Thus, $f-g$ is a trivial morphism, hence $\Sigma f=\Sigma g$. By a dual argument we get that $\Sigma \eta$ is the cokernel of $\Sigma \varepsilon$, hence (2) and (3) are proven.

        Now, let $G \colon \Cal C \to \Cal D$ be another functor satisfying conditions (1), (2) and (3). Then, it is routine to check that  from the assignments $HX:=GX$ and $H \Sigma f:= Gf$ we get a (well-defined) functor $H \colon \Cal C /\Cal Z \to \Cal D$ such that $G=H \Sigma$. The uniqueness of such a functor is clear.
    \end{proof}

\section{Examples}\label{Section_examples}

\subsection{Some pretorsion theories in categories of modules}
Let $R$ be a unital ring. For simplicity, we assume $R$ to be commutative, but similar considerations can be done in the non-commutative setting. Let $\Mod (R)$ denote the category of unital $R$-modules. Given a multiplicatively closed subset $S$ of $R$ (namely a subset $S$ of $R$ such that $1\in S$ and for any $r,s \in S$ one has $rs \in S$), it is possible to define a torsion theory $(\Cal T_S, \Cal F_S)$ where the torsion class consists of those modules $M\in \Mod(R)$ such that $M \otimes_R S^{-1}R=0$ (see \cite[Chapter~VI]{S}; notice that in \cite{S} the term ``pretorsion" is used in a different context). Explicitly, $M\in \Cal T_S$ if, for every $m \in M$, there exists $s \in S$ such that $sm=0$, while $M\in \Cal F_S$ if there are no non-zero elements of $M$ annihilated by elements of $S$.
In view of Proposition~\ref{main}, any inclusion $S\subseteq T$ of multiplicatively closed subsets of $R$ induces a pretorsion theory $(\Cal T_T, \Cal F_S)$ where the class $\Cal Z$ of trivial objects consists of those modules $M$ with the following property: for every $m \in M$, there exists $t \in T$ such that $tm=0$ and if $sm=0$ for some $s\in S$, then $m=0$. In terms of annihilator ideals, for every non-zero $m \in M$ we have $\operatorname{Ann}_R(m)\cap T \neq \emptyset$ and $\operatorname{Ann}_R(m)\cap S=\emptyset$. As a particular case of what we have just seen, any inclusion of prime ideals induces a pretorsion theory, since the complement of a prime ideal is a multiplicatively closed set.

\medskip

The following remark, even if not surprising, has never been pointed out.

\begin{remark} A subcategory $\Cal T$ of a given category $\Cal C$ can be the torsion class of (possibly infinitely) many different pretorsion theories. A way to see it is to take a domain $R$ of infinite Krull dimension. Then, we can consider an infinite chain of prime ideals $0=P_0 \subsetneq P_1 \subsetneq P_2 \subsetneq \dots$ which induces pretorsion theories $(\Cal T_0, \Cal F_i)$  for $i \geq 0$,  where $\Cal T_0$ is the subcategory of ``classical" torsion modules (namely, those modules whose elements are annihilated by a non-zero element of $R$) and $\Cal F_i$ is the full subcategory consisting of those modules $N$ such that for every $n \in N$, $\operatorname{Ann}_R(n)\subseteq P_i$. It is easy to see that $\Cal F_i\neq \Cal F_j$ for $i\neq j$, since $R/P_i \in \Cal F_i \setminus \Cal F_{i-1}$ for every $i \geq 1$.
\end{remark}

\subsection{Pretorsion theories in $\mod (kA_n)$}
Lattices of torsion theories over finite dimensional algebras have been widely studied, see for example \cite{T}. These provide chains of torsion theories and hence can be used to give plenty of examples of pretorsion theories in module categories by applying Theorem~\ref{main}. Recall that any finite dimensional algebra over an algebraically closed field $k$ can be realised as the path algebra of some bound quiver, see for example \cite[Chapter II]{ASS} for details.
For the less familiar reader, given a finite quiver $Q$ (that is an oriented graph), its path algebra over $k$ is the $k$-algebra whose underlying vector space has as basis the set of all paths in $Q$ and the product of two paths is their concatenation whenever this is possible and zero otherwise.

Here, we fix an algebraically closed field $k$ and a positive integer $n$ and focus on a classic example: the path algebra $kA_n$ of the linearly oriented Dynkin quiver

\begin{align*}
A_n: \,\, \stackrel{1}{\bullet}\longrightarrow \stackrel{2}{\bullet}\longrightarrow \dots\longrightarrow \stackrel{n-1}{\bullet}\longrightarrow \stackrel{n}{\bullet.}
\end{align*}

Letting $\mod (kA_n)$ denote the category of finitely generated right $kA_n$-modules, its Auslander-Reiten quiver is shown in Figure \ref{figure_AR}. This is a very useful diagram that collects a lot of key information on $\mod (kA_n)$: its vertices correspond to  the indecomposable modules in the category (that is the building blocks of the objects) and its arrows to the irreducible morphisms in the category (that is the building blocks of the morphisms), see \cite[Chapter IV]{ASS} for details.

    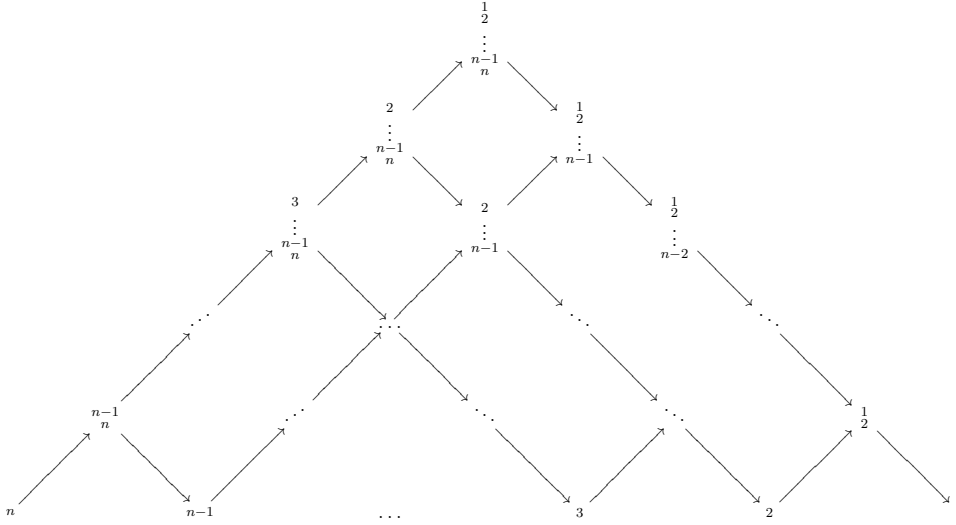
\begin{figure}[ht]
    \scalebox{0.65}{
    \xymatrix @!@C=0.5em@R=0.5em{
    &&&&& {\begin{smallmatrix}1\\2\\\vdots\\n-1\\n \end{smallmatrix}}\ar[rd]&&&\\
    &&&& {\begin{smallmatrix}2\\\vdots\\n-1\\n \end{smallmatrix}}\ar[ru]\ar[rd]&& {\begin{smallmatrix}1\\2\\\vdots\\n-1 \end{smallmatrix}}\ar[rd]&&&\\
    &&&{\begin{smallmatrix}3\\\vdots\\n-1\\n\end{smallmatrix}}\ar[ru]\ar[rd]&&{\begin{smallmatrix}2\\\vdots\\n-1\end{smallmatrix}}\ar[ru]\ar[rd]&&{\begin{smallmatrix}1\\2\\\vdots\\n-2 \end{smallmatrix}}\ar[rd]&&&&\\
    && \reflectbox{$\ddots$} \ar[ru]&&\dots\ar[ru]\ar[rd]&&\ddots\ar[rd]&&\ddots\ar[rd]&&\\
    &{\begin{smallmatrix}n-1\\n\end{smallmatrix}}\ar[ru]\ar[rd]&&\reflectbox{$\ddots$}\ar[ru]&&\ddots\ar[rd]&&\ddots\ar[rd]&&{\begin{smallmatrix}1\\2\end{smallmatrix}}\ar[rd]&&\\
    {\begin{smallmatrix}n\end{smallmatrix}}\ar[ru]&&{\begin{smallmatrix}n-1\end{smallmatrix}}\ar[ru]&&\dots&&{\begin{smallmatrix}3\end{smallmatrix}}\ar[ru]&&{\begin{smallmatrix}2\end{smallmatrix}}\ar[ru]&& {\begin{smallmatrix}1\end{smallmatrix}}
    }}
    \caption{The Auslander-Reiten quiver of $\mod(k A_n)$.}
    \label{figure_AR}
    \end{figure}

\begin{example}
We  give some examples of pretorsion theories in $\mod kA_n$ for any positive integer $n$. In Figure \ref{figure_AR}, the leftmost ascending diagonal consists of all the indecomposable projectives, the descending rightmost diagonal consists of all the indecomposable injectives and all the modules on the bottom row are simple. Moreover, all the ascending arrows are monomorphisms and all the descending ones are epimorphisms, so it is fairly easy to read short exact sequences from the diagram.
    \begin{enumerate}
        \item Consider the chain of torsion classes
        \begin{align*}
            0\subset \mathcal{T}_1\subset \mathcal{T}_2\subset \dots \subset \mathcal{T}_{n-1}\subset\mathcal{T}_n,
        \end{align*}
        where $\mathcal{T}_i:=\text{add}\Bigg\{\begin{smallmatrix} 1 \end{smallmatrix}, \begin{smallmatrix} 1\\2 \end{smallmatrix},\dots, \begin{smallmatrix} 1\\ \vdots \\ i \end{smallmatrix}\Bigg\}$.
        The torsion-free class $\mathcal{F}_i$ corresponding to $\mathcal{T}_i$ is add of the indecomposables in $\mod kA_n$ but not in $\mathcal{T}_i$.
        Then, by Theorem~\ref{main}, we have that for any $i> j$, the pair $(\mathcal{T}_i, \mathcal{F}_j)$ is a pretorsion theory in $\mod (kA_n)$  with class of trivial objects 
        \begin{align*}
        \mathcal{Z}_{i,j}:= \mathcal{T}_i\cap \mathcal{F}_j= \text{add}\Bigg\{ \begin{smallmatrix} 1\\ \vdots \\ j+1 \end{smallmatrix},\begin{smallmatrix} 1\\ \vdots \\ j+2 \end{smallmatrix},\dots, \begin{smallmatrix} 1\\ \vdots \\ i \end{smallmatrix}\Bigg\}.
        \end{align*}
        \item Consider now the chain of torsion classes
        \begin{align*}
            0\subset \mathcal{T}_1\subset \mathcal{T}_2\subset \dots \subset \mathcal{T}_{n-2}\subset\mathcal{T}_{n-1},
        \end{align*}
        with
        \begin{align*}
        \mathcal{T}_i:=\text{add}\Bigg\{ \text{quot}\Bigg\{ \begin{smallmatrix} n\end{smallmatrix},\, \begin{smallmatrix} n-1\\ n\end{smallmatrix}, \dots, \begin{smallmatrix} n-i+1\\ \vdots \\n\end{smallmatrix}\Bigg\}\Bigg\},
        \end{align*}
        where by quot$\{\dots\}$ we mean the quotient closure of the given set. In other words, $\mathcal{T}_1=\text{add}\{\begin{smallmatrix} n\end{smallmatrix}\}$, while for bigger $i$, $\mathcal{T}_i$ is add of the indecomposables lying in the triangular area of the Auslander-Reiten quiver in Figure \ref{figure_AR} delimited by the vertices $\begin{smallmatrix} n\end{smallmatrix},\, \begin{smallmatrix} n-i+1\end{smallmatrix}$ and $\begin{smallmatrix} n-i+1\\ \vdots \\ n\end{smallmatrix}$.
        The torsion-free class $\mathcal{F}_i$ corresponding to $\mathcal{T}_i$ is then
         \begin{align*}
        \mathcal{F}_i:=\text{add}\Bigg\{ \text{submod}\Bigg\{ \begin{smallmatrix} 1\end{smallmatrix},\, \begin{smallmatrix} 1\\ 2\end{smallmatrix}, \dots, \begin{smallmatrix} 1\\\vdots\\n-i\end{smallmatrix}\Bigg\}\Bigg\},
        \end{align*}
         where by submod$\{\dots\}$ we mean the submodule closure of the given set. In other words, $\mathcal{F}_{n-1}=\text{add}\{\begin{smallmatrix} 1\end{smallmatrix}\}$
         while for smaller $i$, $\mathcal{F}_i$ is add of the indecomposables lying in the triangular area of the Auslander-Reiten quiver in Figure~\ref{figure_AR} delimited by the vertices $\begin{smallmatrix} n-i\end{smallmatrix},\, \begin{smallmatrix} 1\end{smallmatrix}$ and $\begin{smallmatrix} 1\\ \vdots \\ n-i\end{smallmatrix}$. Then, by Theorem~\ref{main}, we have that for any $i<j$, the pair $(\mathcal{T}_i, \mathcal{F}_j)$ is a pretorsion theory in $\mod (kA_n)$  with class of trivial objects 
        \begin{align*}
        \mathcal{Z}_{i,j}:= \mathcal{T}_i\cap \mathcal{F}_j= \text{add}\Bigg\{\text{quot}\Bigg\{ \begin{smallmatrix} n-j \end{smallmatrix},\begin{smallmatrix} n-j-1\\n-j \end{smallmatrix},\dots, \begin{smallmatrix} n-i+1\\ \vdots \\ n-j \end{smallmatrix}\Bigg\}\Bigg\}.
        \end{align*}
        \end{enumerate}
\end{example}

\begin{remark}\label{Remark_A2}
    Theorem~\ref{main} gives a way to construct many pretorsion theories but not all pretorsion theories can be obtained in this way. In particular, Theorem~\ref{main} always produces a pretorsion theory where the first half is a torsion class and the second a torsionfree class in the classical sense.
    Even in the case of abelian categories, pretorsion theories do not have such restrictive properties, and one can use Theorem~\ref{extension_Serre_2} to produce more examples.
    Consider, for instance, the path algebra $kA_2$. Then, $\mod (kA_2)$ contains exactly three indecomposable modules and it has Auslander-Reiten quiver

    \begin{align*}
    \xymatrix @C=1em@R=1em{
    & {\begin{smallmatrix}1\\2\end{smallmatrix}}\ar[rd]^{\beta}&\\
    {\begin{smallmatrix}2\end{smallmatrix}}\ar[ru]^{\alpha}&& {\begin{smallmatrix}1\end{smallmatrix},}
    }
    \end{align*}
    where ${\begin{smallmatrix}2\end{smallmatrix}}$ is a simple projective, ${\begin{smallmatrix}1\\2\end{smallmatrix}}$ a projective-injective and ${\begin{smallmatrix}1\end{smallmatrix}}$ a simple injective.
    Consider the torsion pair $(\mathcal{U}=\text{add}\,\{\begin{smallmatrix}1\end{smallmatrix},\begin{smallmatrix}1\\2\end{smallmatrix}\},\mathcal{V}=\text{add}\,\{\begin{smallmatrix}2\end{smallmatrix}\})$ and the Serre subcategory $\mathcal{S}=\text{add}\,\{\begin{smallmatrix}1\end{smallmatrix}\}$ in $\mod(kA_2)$. Applying Theorem~\ref{extension_Serre_2}, we get the pretorsion theory 
    \begin{align*}
        (\mathcal{U}\ast \mathcal{S}=\text{add}\,\{\begin{smallmatrix}1\end{smallmatrix},\begin{smallmatrix}1\\2\end{smallmatrix}\},\mathcal{S}\ast\mathcal{V}=\text{add}\,\{\begin{smallmatrix}2\end{smallmatrix}, \begin{smallmatrix}1\end{smallmatrix}\}),
    \end{align*}
    with class of trivial objects $\mathcal{S}$. It is easy to verify that the short $\mathcal{S}$-exact sequences of the three indecomposable modules are
    \begin{align*}
    \xymatrix @C=1.5em@R=1em{
    \begin{smallmatrix}0\end{smallmatrix}\ar[r]&\begin{smallmatrix}2\end{smallmatrix}\ar@{=}[r]& \begin{smallmatrix}2\end{smallmatrix},
    &{\begin{smallmatrix}1\\2\end{smallmatrix}}\ar@{=}[r]&{\begin{smallmatrix}1\\2\end{smallmatrix}}\ar[r]^{\beta} &\begin{smallmatrix}1\end{smallmatrix},
    &\begin{smallmatrix}1\end{smallmatrix}\ar@{=}[r]&\begin{smallmatrix}1\end{smallmatrix}\ar@{=}[r]&\begin{smallmatrix}1\end{smallmatrix}.
    }
    \end{align*}
    
    Note that the above pretorsion theory cannot be obtained by applying Theorem~\ref{main}. In fact, the short exact sequence
    \begin{align*}
    \xymatrix @C=1.5em@R=1em{
        0\ar[r]& \begin{smallmatrix}2\end{smallmatrix}
        \ar[r]^-{\alpha}
        &{\begin{smallmatrix}1\\2\end{smallmatrix}}
        \ar[r]^-{\beta} &\begin{smallmatrix}1\end{smallmatrix}\rightarrow 0}
    \end{align*}
    is such that the end-terms are in 
    $\mathcal{S}\ast\mathcal{V}$, while the middle term is not. Hence $\mathcal{S}\ast\mathcal{V}$ is not closed under extensions and so it is not a torsionfree class. Note that, however, $\mathcal{S}\ast\mathcal{V}$ is still closed under $\mathcal{S}$-extensions.
\end{remark}

\subsection{Pretorsion theories and recollements}\label{recollements}
Let $\Cal C$ be an abelian category. Given a Serre subcategory $\Cal S$, it is possible to construct an abelian quotient category $j^*\colon \Cal C \to [\Cal C / \Cal S]$ (see \cite{Gabriel}), where $j^*$ is an essentially surjective exact functor whose kernel is $\Cal S$. Note that this quotient category construction is different from the one from Section~\ref{stable_category}.
The category $\Cal S$ is called bilocalising if $j^*$ has both a left adjoint $j_!$ and a right adjoint $j_*$. In this case, $\Cal S$ turns out to be both a monocoreflective and an epireflective subcategory of $\Cal C$ and thus induces a recollement of $\Cal C$ \cite[Remark~2.8]{Jorge}:
$$
\xymatrix@!=50pt{
\Cal S \ar[r]^{i_*} & \Cal C \ar[r]^{j^*} \ar@/_15pt/[l]_{i^*} \ar@/^15pt/[l]_{i^!} & [\Cal C /\Cal S] \ar@/_15pt/[l]_{j_!} \ar@/^15pt/[l]_{j_*}
}
$$

Notice that, up to equivalence, any recollement of abelian categories arises in this way (see \cite[Theorem~4.1]{Jorge} for a precise statement of this fact). A torsion theory $(\Cal U, \Cal V)$ in an abelian category can be ``extended" (in the sense of Theorem~\ref{extension_Serre_2}) by any bilocalising subcategory $\Cal S$ to a pretorsion theory $(\Cal U \ast \Cal S, \Cal S \ast \Cal V)$.  Notice that if we apply the exact functor $j^*$ to diagram~(\ref{diagram}) in the proof of Theorem~\ref{extension_Serre_2}, we have that if
    $$
    \xymatrix{T_X \ar[r] & X  \ar[r] & F_X}
    $$
    is an $\Cal S$-short exact sequence of $X \in \Cal C$, with $T_X \in \Cal U \ast \Cal S $ and $F_X \in \Cal S \ast \Cal V$, then
    $$
    \xymatrix{0 \ar[r] & j^*(T_X) \ar[r] & j^*(X)  \ar[r] & j^*(F_X) \ar[r] & 0}
    $$
    is a short exact sequence in $[\Cal C/ \Cal S] $. Nevertheless, neither the image of $(\Cal U \ast \Cal S, \Cal S \ast \Cal V)$ in the quotient $[\Cal C / \Cal S]$ nor that of $(\Cal U, \Cal V)$  are torsion theories in general. To see this, consider the torsion theory $(\Cal U, \Cal V)=(\text{add}\,\{\begin{smallmatrix}1\end{smallmatrix}\},\text{add}\,\{\begin{smallmatrix}2\end{smallmatrix}, \begin{smallmatrix}1\\2\end{smallmatrix}\})$ in $\mod (kA_2)$ (see Remark \ref{Remark_A2}) and take the Serre subcategory $\Cal S=\text{add}\,\{\begin{smallmatrix}2\end{smallmatrix}\}$. Then, it suffices to observe that $\begin{smallmatrix}1\end{smallmatrix}\cong \begin{smallmatrix}1\\2\end{smallmatrix}$ in $[\mod (kA_2)/\Cal S]$.
    
    On the other hand, using the construction from Theorem~\ref{stable_theorem}, the functor $\Sigma$ sends the pretorsion theory $(\Cal U\ast \Cal S=\text{add}\,\{\begin{smallmatrix}1\end{smallmatrix}, \begin{smallmatrix}2\end{smallmatrix}\},\Cal S\ast \Cal V=\text{add}\,\{\begin{smallmatrix}2\end{smallmatrix}, \begin{smallmatrix}1\\2\end{smallmatrix}\})$ in $\mod (kA_2)$ to the torsion theory $(\text{add}\,\{\begin{smallmatrix}1\end{smallmatrix}\},\text{add}\,\{\begin{smallmatrix}1\\2\end{smallmatrix}\})$ in $\mod (kA_2)/\Cal S$. Notice that here $\mod (kA_2)/\Cal S$ is an additive category that is not abelian.

\subsection{Non-epireflective Serre subcategories.}
The following example shows that the assumptions on Proposition~\ref{extension_Serre_1} do not guarantee that $(\mathcal{U}\ast \mathcal{S}, \mathcal{S}\ast \mathcal{V})$ is a pretorsion theory even in the abelian case.
\begin{example}\label{injective-reduced} 
Consider the abelian category $\Mod (\mathbb{Z})$ with torsion theory $(\mathcal{U},\mathcal{V})$, where $\mathcal{U}$ is the class of injective abelian groups and $\mathcal{V}$ the class of reduced abelian groups \cite[Example~1.13.6]{Borceux}. Take $\mathcal{S}$ to be the class of torsion abelian groups and note that this is a Serre subcategory that is not an epireflective subcategory of $\Mod (\mathbb{Z})$.
We show that $(\mathcal{U}\ast \mathcal{S}, \mathcal{S}\ast \mathcal{V})$ is not a pretorsion theory by showing that $\mathcal{S}\ast \mathcal{V}$ is not closed under products.
For example, take the injective abelian group $\mathbb{Q}/\mathbb{Z}$. This is clearly a torsion abelian group, hence it belongs to $\mathcal{S}\ast \mathcal{F}$. 
Now, consider the product $\prod_{i\in\mathbb{N}} \mathbb{Q}/\mathbb{Z} $ of infinitely many copies of $\mathbb{Q}/\mathbb{Z}$, which is injective but not torsion. Suppose for a contradiction that $\prod_{i\in\mathbb{N}} \mathbb{Q}/\mathbb{Z}$ is in $\mathcal{S}\ast \mathcal{F}$, that is, there is a short exact sequence of the form
\begin{align*}
    0\rightarrow S \rightarrow \prod_{i\in\mathbb{N}} \mathbb{Q}/\mathbb{Z} \rightarrow F \rightarrow 0
\end{align*}
with $S\in\mathcal{S}$ and $F\in \mathcal{F}$. Since $F$ is a quotient of an injective $\mathbb{Z}$-module, then it is injective. Thus $F=0$ and $S\cong\prod_{i\in\mathbb{N}} \mathbb{Q}/\mathbb{Z}$, contradicting the latter not being torsion.
\end{example}    

\subsection{Internal groupoids in a homological category} Let $\Cal C$ be a homological category (that is, a finitely complete, regular and protomodular category with a zero object \cite{BB}) and consider the category $\text{Grpd}(\Cal C)$ of internal groupoids in $\Cal C$. In \cite{BG}, two torsion theories in $\text{Grpd}(\Cal C)$ are presented, namely $(\text{Ab}(\Cal C), \text{Eq}(\Cal C))$ and $(\text{ConnGrpd}(\Cal C), \Cal C)$, where $\text{Ab}(\Cal C), \text{Eq}(\Cal C)$ and $\text{ConnGrpd}(\Cal C)$ denote the subcategories of (internal) abelian objects, equivalent relations and connected groupoids respectively, while any object $X$ of $\Cal C$ can be seen as an internal ``trivial" groupoid taking the identities $\quad\ \xymatrix{X \ar@(ld,lu) \ar@<-.7ex>[r] \ar@<.7ex>[r] & X \ar[l] }$. Since every internal abelian object is a connected groupoid, then
$(\text{ConnGrpd}(\Cal C), \text{Eq}(\Cal C))$  is a pretorsion theory in $\text{Grpd}(\Cal C)$ by Theorem~\ref{main}.

\subsection{Chains of torsion theories for complexes}
Let $\Cal C$ be a pointed regular category where every regular epimorphism is a normal epimorphism. Consider the category $ch(\Cal C)$ of chain complexes in $\Cal C$. We shall denote a generic object of $ch(\Cal C)$ by
$$\xymatrix{
X_{\bullet}:  &\dots \ar[r] & X_{n+1} \ar[r]^{\delta_{n+1}} & X_n \ar[r]^{\delta_{n}}&  X_{n-1} \ar[r]^{\delta_{n-1}} & \dots & n \in \mathbb{Z}
}$$

In \cite{Cafaggi}, several chains of torsion theories in $ch(\Cal C)$ have been studied. Here we only present one example, in order to apply Theorem~\ref{main} explicitly. For every $n \in \mathbb{Z}$, there is a torsion theory $(\Cal T_n, \Cal F_n)$, where $\Cal T_n$ consists of those chains $X_{\bullet}$ with $X_k=0$ for all $k \leq n$, while $\Cal F_n$ consists of those chains $X_{\bullet}$ with $X_k=0$ for all $k > n$ and $\delta_n$ is a monomorphism. Thus, for every $n \geq m$ we get a pretorsion theory $(\Cal T_m, \Cal F_n)$ where the class $\Cal Z_{m,n}$ of trivial objects consists of those chains $X_\bullet$ with $X_k=0$ for all $k\leq m$ and $k > n$, and $\delta_n$ is a monomorphism.

\subsection{Pretorsion theories and stability functions}
Let $\Cal C$ be an abelian length category, that is, an essentially small abelian category such that every object has a finite composition series. A stability function $\Phi$ is a map from the non-zero objects of $\Cal C$ into a totally ordered set $(P, \leq)$ satisfying the following properties (see \cite{BST} and the references therein):\\
(1) if $A \cong B$ for some non-zero objects $A,B \in \Cal C$, then $\Phi(A)=\Phi(B)$;\\
(2) if $0\to A\to B\to C \to 0$ is a short exact sequence of non-zero objects of $\Cal C$, then exactly one of the following three cases can occur:
$$
\bullet \ \Phi(A)< \Phi(B)< \Phi(C); \qquad \bullet \ \Phi(A)> \Phi(B)> \Phi(C); \qquad \bullet \ \Phi(A)= \Phi(B)=\Phi(C).
$$

For every $p \in P$, there is a torsion theory $(\Cal T_{\geq p}, \Cal F_{<p})$ \cite[Proposition~2.19]{BST}, where:\\
\indent $\Cal T_{\geq p}:=\{ X \in \Cal C \mid \Phi(Y)\geq p \text{ for every quotient } Y \text{ of } X\}\cup \{0\}$ and \\
\indent $\Cal F_{< p}:=\{ X \in \Cal C \mid \Phi(H)< p \text{ for every subobject } H \text{ of } X\}\cup \{0\}$\\
(the obvious variation with $(\Cal T_{>p}, \Cal F_{\leq p})$ holds as well) \cite[Section~2]{BST}. Then, for every $p,q \in P$ with $p \leq q$, there is a pretorsion theory $(\Cal T_{\geq p}, \Cal F_{< q})$.

\section*{Acknowledgement}
The authors would like to thank Marino Gran and Jorge Nuno Dos Santos Vitória for their useful comments and suggestions.

\end{document}